\documentclass[12pt]{amsart}
\usepackage{amscd,amssymb,times, amsmath, fullpage}
\usepackage{graphicx}
\usepackage{enumitem}
\newtheorem{theorem}{Theorem}%[section]

\newtheorem{proposition}[theorem]{Proposition}
\newtheorem{corollary}[theorem]{Corollary}
\theoremstyle{definition}
\newtheorem{definition}[theorem]{Definition}

\newtheorem{example}[theorem]{Example}

\theoremstyle{remark}
\newtheorem{remark}[theorem]{Remark}
\newcommand{\R}{\mathbb{ R}}
\newcommand{\HH}{\mathbb{ H}}

\newcommand{\FF}{\mathcal{ F}}

\newcommand\Curv{\operatorname{R}}

\newcommand\tr{\operatorname{trace}}
\newcommand\h{\operatorname{h}}

\newcommand\B{\operatorname{B}}

\newcommand\Ker{\operatorname{Ker}}

\begin{document}

\title{Minimality of invariant submanifolds in Metric Contact Pair Geometry}
\author{Gianluca Bande}
\address{Dipartimento di Matematica e Informatica, Universit{\`a} degli Studi di Cagliari, Via Ospedale 72, 09124 Cagliari, Italy}
\email{gbande{\char'100}unica.it}
\author{Amine Hadjar}
\address{Laboratoire de Math{\'e}matiques, Informatique et
Applications, Universit{\'e} de Haute Alsace - 4, Rue des
Fr{\`e}res Lumi{\`e}re, 68093 Mulhouse, France}
\email{mohamed.hadjar{\char'100}uha.fr}

%%%%%%%%%%%%%%%%%%%%
\thanks{The first author was supported by P.R.I.N. 2010/11 -- Variet\`{a}  reali e complesse: geometria, topologia e analisi armonica -- Italy.}
\date{\today; MSC 2010 classification: primary 53C25; secondary 53B20, 53D10, 53B35, 53C12}
\keywords{Contact Pair; Vaisman manifold; Invariant submanifold; Minimal submanifold.}

\maketitle

%\thanks{The first author was supported by P.R.I.N. 2010/11 -- Variet\`{a}  reali e complesse: geometria, topologia e analisi armonica -- Italy.}
%\date{\today; MSC 2010 classification: primary 53C25; secondary 53B20, 53C12, 53B35}

\begin{abstract}
We study invariant submanifolds of manifolds endowed with a normal or complex metric contact pair with decomposable %structure tensor 
endomorphism field %
$\phi$. For the normal case, we prove that a $\phi$-invariant submanifold tangent to a  Reeb vector field and orthogonal to the other one is minimal. For a $\phi$-invariant submanifold $N$ everywhere transverse to both the Reeb vector fields but not orthogonal to them, we prove that it is minimal if and only if the angle between the tangential component $\xi$ (with respect to $N$) of a Reeb vector field and the Reeb vector field itself is constant along the integral curves of $\xi$. For the complex case (when just one of the two natural almost complex structures is supposed to be integrable), we prove that a complex submanifold is minimal if and only if it is tangent to both the Reeb vector fields.
%\keywords{Contact Pair \and Vaisman manifold \and invariant submanifold \and minimal submanifold.}
%\subclass{primary 53C25; secondary 53B20, 53D10, 53B35, 53C12}
\end{abstract}

\section{Introduction}
It is well known that on a K\"ahler manifold the $J$-invariant submanifolds ($J$ being the complex structure of the K\"ahler manifold) are minimal. On the other hand for a Sasakian manifold, and more generally for a contact metric manifold, a $\phi$-invariant submanifold is also minimal, where $\phi$ is the %structure tensor 
endomorphism field %
of the contact metric structure. Similar results are known for a special class of Hermitian manifolds, that is the class of locally conformally K\"ahler (lcK) manifolds and in particular for the subclass of Vaisman manifolds. Dragomir and Ornea \cite[Theorem 12.1]{Ornea} have shown that a $J$-invariant submanifold of an lcK manifold is minimal if and only if the submanifold is tangent to the Lee vector field (and therefore tangent to the anti-Lee vector field). In fact this result is a slight generalization of the following result of Vaisman \cite{Vaisman}: a $J$-invariant submanifold of a generalized Hopf manifold (nowadays called Vaisman manifold) inherits a generalized Hopf manifold structure if and only if it is minimal (or, equivalently, if and only if the submanifold is tangent to the Lee vector field). In \cite{BK2} it was shown that the notion of non-K\"ahler Vaisman manifold, after constant rescaling of the metric, is equivalent to the notion of {\it normal metric contact pair} \cite{BH3} of type $(h,0)$ and the Lee and anti-Lee vector fields correspond to the Reeb vectors fields of the pair. Moreover this equivalence enlightened the fact that on a Vaisman manifold there is another complex structure $T$ with opposite orientation with respect to $J$. In terms of normal metric contact pairs the generalization of Vaisman's result can be stated as follows: a $J$-invariant submanifold of a normal metric contact pair manifold of type $(h,0)$ is minimal if and only if the submanifold is tangent to the Reeb vector fields or, equivalently, if it is also $T$-invariant. These observations lead to the study of the invariant submanifolds of normal metric contact pairs of type $(h,k)$ \cite{BH3} also called {\it Hermitian bicontact structures}  \cite{Blair2}.

More precisely, recall that a \emph{metric contact pair} \cite{BH2} of type $(h,k)$ on a manifold $M$ is $4$-tuple $(\alpha_1, \alpha_2, \phi, g)$ such that $(\alpha_1, \alpha_2)$ is a contact pair \cite{Bande1,BH} of type $(h,k)$, $\phi$ is an endomorphism field of $M$ such that
\begin{eqnarray*}
\phi^2=-Id + \alpha_1 \otimes Z_1 + \alpha_2 \otimes Z_2 , \; \;
\phi Z_1=\phi Z_2 =0 ,
\end{eqnarray*}
where $Z_1$ and $Z_2$ are the Reeb vector fields of $(\alpha_1 ,
\alpha_2)$,
and $g$ is a Riemannian metric such that $g(X, \phi Y)= (d \alpha_1 + d
\alpha_2) (X,Y)$ and $g(X, Z_i)=\alpha_i(X)$, for $i=1,2$. The metric contact pair is said to be {\it normal} \cite{BH3} if the two almost complex structures of opposite orientations $J=\phi  - \alpha_2 \otimes Z_1 + \alpha_1 \otimes Z_2 $ and $T=\phi  + \alpha_2 \otimes Z_1 - \alpha_1 \otimes Z_2 $ are integrable. A quite important notion is the one of {\it decomposability} of $\phi$, which means that the tangent spaces of the leaves of the characteristic foliations of the pair are preserved by $\phi$. The decomposability of $\phi$ is equivalent to the orthogonality of the two characteristic foliations and implies that the their leaves are $\phi$-invariant submanifolds and moreover minimal \cite{BH4}.

In this paper, after giving some characterizations of normal metric contact pairs with decomposable $\phi$, we address the problem of the minimality of the invariant submanifolds. Observe that on a metric contact pair manifold we have several notions of invariant submanifold: with respect to $\phi$, to $J$ or to $T$. We first give some general results concerning the invariant submanifolds of a metric contact pair manifold with decomposable $\phi$, then we specialize to the normal case and we prove the following:

\bigskip
{\bf Theorem \ref{th:phi-invariant Z_1 normal}.}
{\it Let $(M, \alpha_1, \alpha_2, \phi, g)$ be a normal metric contact pair manifold with decomposable $\phi$ and Reeb vector fields $Z_1$ and $Z_2$.
If $N$ is a $\phi$-invariant submanifold of $M$ such that $Z_1$ is tangent and $Z_2$ orthogonal to $N$, then $N$ is minimal. Moreover if $N$ is connected, then it is a Sasakian submanifold of one of the Sasakian leaves of the characteristic foliation of $\alpha_2$.}

\bigskip
{\bf Theorem \ref{th:phi-invariant-transverse}.}
{\it Let $(M, \alpha_1, \alpha_2, \phi, g)$ be a normal metric contact pair manifold with decomposable $\phi$ and Reeb vector fields $Z_1$ and $Z_2$.
Let $N$ be a $\phi$-invariant submanifold of $M$ nowhere tangent and nowhere orthogonal to $Z_1$ and $Z_2$.
Then $N$ is minimal if and only if the angle between $Z_1^T$ ($Z_1^T$ being the tangential component of $Z_1$ along $N$) and $Z_1$ (or equivalently $Z_2$) is constant along the integral curves of $Z_1^T$.}

\bigskip
{\bf Theorem \ref{th-j-integrable-minimal}.}
{\it Let $(M, \alpha_1, \alpha_2, \phi, g)$ be a metric contact pair manifold with decomposable $\phi$ and Reeb vector fields $Z_1$ and $Z_2$. Suppose that the almost complex structure  $J=\phi  - \alpha_2 \otimes Z_1 + \alpha_1 \otimes Z_2 $ is integrable. Then a $J$-invariant submanifold $N$ of $M$ is minimal if and only if it is tangent to the Reeb distribution.}

\bigskip
The last result, applied to the case of a normal metric contact pair, gives the desired generalization of the result of Vaisman to normal metric contact pairs of type $(h,k)$. Nevertheless it should be remarked that the full generalization of the original Vaisman result concerning the Vaisman manifolds is not true. In fact we give an example where the submanifold is both $J$ and $T$-invariant, then tangent to the Reeb distribution, and therefore minimal, but it does not inherit the contact pair structure of the ambient manifold.

In what follows we denote by $\Gamma(B)$ the space of sections of
a vector bundle $B$. For a given foliation $\mathcal{F}$ on a
manifold $M$, we denote by $T\mathcal{F}$ the subbundle of $TM$
whose fibers are given by the distribution tangent to the leaves.
All the differential objects considered are assumed to be smooth.

\section{Preliminaries}\label{preliminaries}
%%%%%%%%%%%%%%%%%%%%%%%%%%%%%%%%%%%%%%%%%%%%%%%%%%%%%%%%%%%%%%%%%%%%%%%%%%%
A \emph{contact pair} (or \emph{bicontact structure}) \cite{Bande1,BH,Blair2} of type $(h,k)$ on a manifold $M$ is a pair $(\alpha_1, \alpha_2)$ of 1-forms such that:
\begin{eqnarray*}
&\alpha_1\wedge (d\alpha_1)^{h}\wedge\alpha_2\wedge
(d\alpha_2)^{k} \;\text{is a volume form},\\
&(d\alpha_1)^{h+1}=0 \; \text{and} \;(d\alpha_2)^{k+1}=0.
\end{eqnarray*}

The \'Elie Cartan characteristic classes of $\alpha_1$ and $\alpha_2$ are constant and equal to $2h+1$ and $2k+1$ respectively.
The distribution $\Ker \alpha_1 \cap \Ker
d\alpha_1$ (respectively
 $\Ker \alpha_2 \cap \Ker d\alpha_2$) is completely integrable \cite{Bande1,BH} and then it determines the \emph{characteristic
  foliation} $\mathcal{F}_1$ of $\alpha_1$ (respectively $\mathcal{F}_2$ of $\alpha_2$) whose leaves are endowed with a contact form induced by $\alpha_2$ (respectively $\alpha_1$).
The equations
\begin{eqnarray*}
&\alpha_1 (Z_1)=\alpha_2 (Z_2)=1  , \; \; \alpha_1 (Z_2)=\alpha_2
(Z_1)=0 \, , \\
&i_{Z_1} d\alpha_1 =i_{Z_1} d\alpha_2 =i_{Z_2}d\alpha_1=i_{Z_2}
d\alpha_2=0 \, ,
\end{eqnarray*}
where $i_X$ is the contraction with the vector field $X$, determine uniquely the two vector fields $Z_1$ and $Z_2$, called \emph{Reeb vector fields}. Since they commute \cite{Bande1,BH}, they give rise to a locally free $\mathbb{R}^2$-action, %, called  the \emph{Reeb action}, 
an integrable distribution called \emph{Reeb distribution},
and then a foliation $\mathcal{R}$ of $M$ by surfaces.
The tangent bundle of $M$
can be split as:
$$
TM=T\mathcal F _1 \oplus T\mathcal F _2 =\mathcal{H}_1 \oplus
\mathcal{H}_2 \oplus \mathcal{V}  ,
$$
where
$T\mathcal F _i$
is the subbundle determined by the characteristic foliation
$\mathcal F _i$, $\mathcal{H}_i$ the subbundle whose fibers are
given by $\ker d\alpha_i \cap \ker \alpha_1 \cap \ker \alpha_2$,
 $\mathcal{V} =\mathbb{R} Z_1 \oplus \mathbb{R} Z_2$ 
and $\mathbb{R} Z_1, \mathbb{R} Z_2$ the line bundles determined
by the Reeb vector fields.
Moreover we have $T\mathcal F _1=\mathcal{H}_1 \oplus \mathbb{R}
Z_2 $ and $T\mathcal F _2=\mathcal{H}_2 \oplus \mathbb{R} Z_1 $.
The fibers of the subbundle $\mathcal{H}_1 \oplus \mathcal{H}_2$ are given by the distribution $\ker \alpha_1 \cap \ker \alpha_2$.
\begin{definition}
We say that a vector field is \emph{vertical} if it is a section of $\mathcal{V}$ and \emph{horizontal} if it is a section of $\mathcal{H}_1 \oplus
\mathcal{H}_2$. 
A tangent vector will be said \emph{vertical} if it lies in $\mathcal{V}$ and \emph{horizontal} if it lies in $\mathcal{H}_1 \oplus
\mathcal{H}_2$. 
The subbundles $\mathcal{V}$ and $\mathcal{H}_1 \oplus \mathcal{H}_2$ will be called \emph{vertical} and \emph{horizontal} respectively.
\end{definition}

The two distributions $\ker d\alpha_1$ and $ \ker d\alpha_2$ are also completely integrable and give rise to the \emph{characteristic foliations} $\mathcal{G}_i$ of $d\alpha_i$ respectively. We have 
$
T\mathcal{G} _i =\mathcal{H}_i \oplus \mathcal{V}  ,
$
for $i=1,2$. The contact pair on $M$ induces on each leaf of $\mathcal{G}_1$ (respectively of $\mathcal{G}_2$) a contact pair of type $(0,k)$ (respectively $(h,0)$).
Each of them is foliated by leaves of $\mathcal F _1$ (respectively of $\mathcal F _2$) and also by leaves of $\mathcal{R}$.

A \emph{contact pair structure} \cite{BH2} on a manifold $M$ is a triple
$(\alpha_1 , \alpha_2 , \phi)$, where $(\alpha_1 , \alpha_2)$ is a
contact pair and $\phi$ a tensor field of type $(1,1)$ such that:

\begin{eqnarray*}
\phi^2=-Id + \alpha_1 \otimes Z_1 + \alpha_2 \otimes Z_2 , \; \;
\phi Z_1=\phi Z_2 =0 ,
\end{eqnarray*}
where $Z_1$ and $Z_2$ are the Reeb vector fields of $(\alpha_1 ,
\alpha_2)$.

One can see that $\alpha_i \circ \phi =0$, for $i=1,2$ and that the rank of $\phi$ is
equal to $\dim M -2$. The endomorphism $\phi$ is said to be \emph{decomposable}  if
$\phi (T\mathcal{F}_i) \subset T\mathcal{F}_i$, for $i=1,2$.%, then $(\alpha_i , Z_i ,\phi)$ induces, on every contact leaf of $\mathcal{F}_j$ for $j\neq i$, an almost contact structure.

In \cite{BH3} we defined the notion of \emph{normality} for a contact pair structure
as the integrability of the two natural almost complex structures of opposite orientations $J=\phi  - \alpha_2 \otimes Z_1 + \alpha_1 \otimes Z_2 $ and $T=\phi  + \alpha_2 \otimes Z_1 - \alpha_1 \otimes Z_2 $ on $M$.
This is equivalent to the vanishing of the tensor field
\begin{eqnarray*}
N^1 (X,Y)=  [\phi , \phi ](X, Y) +2 d\alpha_1 (X,Y) Z_1 +2 d\alpha_2 (X,Y) Z_2 \, ,
\end{eqnarray*}
where $[\phi , \phi ]$ is the Nijenhuis tensor of $\phi$. 

A \emph{compatible} metric \cite{BH2} with respect to a contact pair structure $(\alpha_1 , \alpha_2 ,\phi )$ on a manifold $M$,
 with Reeb vector fields $Z_1$ and $Z_2$
is a Riemannian metric $g$ on $M$ such that
 $g(\phi X,\phi Y)=g(X,Y)-\alpha_1 (X)
\alpha_1 (Y)-\alpha_2 (X) \alpha_2 (Y)$ for all $X,Y \in \Gamma
(TM)$. A Riemannian metric $g$ is said to be an \emph{associated} metric \cite{BH2} if
$g(X, \phi Y)= (d \alpha_1 + d
\alpha_2) (X,Y)$ and $g(X, Z_i)=\alpha_i(X)$, for $i=1,2$ and for
all $X,Y \in \Gamma (TM)$.

It is clear that an associated metric is compatible, but the converse is not true. However a compatible metric always satisfies the second equation 
$g(X, Z_i)=\alpha_i(X)$, for $i=1,2$, and then the subbundles $\mathcal{H}_1 \oplus \mathcal{H}_2$, $\mathbb{R} Z_1$, $\mathbb{R} Z_2$ are pairwise orthogonal.

A \emph{metric contact pair (MCP)} on a manifold $M$ is a
$4$-tuple $(\alpha_1, \alpha_2, \phi, g)$ where $(\alpha_1,
\alpha_2, \phi)$ is a contact pair structure and $g$ an associated
metric with respect to it. The manifold $M$ will be called an \emph{MCP manifold} or an \emph{MCP} for short.

For an MCP $(\alpha_1, \alpha_2, \phi, g)$ the endomorphism field $\phi$ is decomposable
if and only if the characteristic foliations $\mathcal{F}_1$, 
$\mathcal{F}_2$ are orthogonal \cite{BH2}. In this case $(\alpha_i, \phi, g)$ induces a metric contact structure
on the leaves of $\mathcal{F}_j$ , for $j\neq i$. Also, the MCP induces MCP's on the leaves of $\mathcal{G}_i$.

It has been shown in \cite{BK2} that a normal MCP structure of type $(h,0)$ is nothing but a non-K\"ahler Vaisman structure on the manifold.

If the MCP on $M$ is normal with decomposable endomorphism, then the leaves of $\mathcal{F}_i$ are Sasakian. Also,  those of $\mathcal{G}_i$ are non-K\"ahler Vaisman manifolds foliated by leaves of $\mathcal{F}_i$ (which are Sasakian) and by leaves of $\mathcal{R}$ (which are complex curves). 
%If in addition the MCP is normal these induced structures are Sasakian.% on each leaf.

Interesting examples and properties of such structures were given in \cite{BB,BBH,BH,BH2,BH3,BH4}.
%Here is the simplest one

\begin{example}\label{exampledarboux}
If $(M_1, \alpha_1, \phi_1, g_1)$ and $(M_2, \alpha_2, \phi_2, g_2)$ are two Sasakian manifolds, then the structure $(\alpha_1,\alpha_2, \phi, g)$ with $\phi=\phi_1\oplus \phi_2$ and $g=g_1\oplus g_2$ is a normal MCP on the product $M_1\times M_2$ with decomposable endomorphism.
So we have such a structure on $\mathbb{R}^{2h+2k+2}$ using the standard Sasakian structures on $\mathbb{R}^{2h+1}$ and 
$\mathbb{R}^{2k+1}$ given by 

\begin{eqnarray*}
&\alpha_1=\frac{1}{2}(dz-\sum_{i=1}^h y_{i}dx_{i}) , \; \; g_1=\alpha_1\otimes \alpha_1+\frac{1}{4}\sum_{i=1}^h ((dx_{i})^2+(dy_{i})^2 )\\
&\alpha_2=\frac{1}{2}(dz^{\prime}-\sum_{j=1}^k y^{\prime}_{j}dx^{\prime}_{j}), \; \; g_2=\alpha_2\otimes \alpha_2+\frac{1}{4}\sum_{j=1}^k ((dx^{\prime}_{j})^2+(dy^{\prime}_{j})^2 ).
\end{eqnarray*}
The Reeb vector fields are $Z_1=2\frac{\partial}{\partial z}$ and $Z_2=2\frac{\partial}{\partial z^{\prime}}$. The endomorphism $\phi$ sends the vector fields 
$X_i=\frac{\partial}{\partial y_i}$ to $X_{h+i}=\frac{\partial}{\partial x_i}+y_{i}\frac{\partial}{\partial z}$ 
and $X^{\prime}_j=\frac{\partial}{\partial y^{\prime}_j}$ to $X^{\prime}_{k+j}=\frac{\partial}{\partial x^{\prime}_j}+y^{\prime}_{j}\frac{\partial}{\partial z^{\prime}}$.
\end{example} 

\begin{remark}
As already explained in \cite[Section 3.4]{BH3}, normal MCP manifolds were already studied in \cite{Blair2} under the name of {\it bicontact Hermitian manifolds} and can be regarded as a generalization of the Calabi-Eckmann manifolds.  An MCP is a special case of a {\it metric $f$-structure with complemented frames} in the sense of Yano \cite{yano}. The normality condition for such structures is well known and is in fact the same condition we have asked for an MCP to be normal. What is completely new in our context is the fact that the normality condition is equivalent to the integrability of the two almost complex structures $J$ and $T$ defined above. Even in the special case of the Vaisman manifolds this was not known as it was observed in \cite{BK2} (see the short discussion before Proposition 2.10) where it was used for classification purposes. It should also be observed that $\mathcal{P}$-{\it manifolds} introduced in \cite{Vaisman2} by Vaisman are MCP manifolds of type $(h,0)$ with Killing Reeb vector fields.
\end{remark}

%%%%%%%%%%%%%%%%%%%%%%%%%%%%

%%%%%%%%%%%%%%%%%%%%%%%%%%%%%%%%%%%%%%%%%%%%%
%
%  Verifier si l'une des proporietes suivantes est utilisee 
%
%%%%%%%%%%%%%%%%%%%%%%%%%%%%%%%%%
%\begin{theorem} Let $M$ be a manifold endowed with a contact pair structure $(\alpha_1 , \alpha_2, \phi)$, with Reeb vector fields $Z_1 , Z_2$. Let $g$ be an associated metric with respect to it, with Levi-Civita connection $\nabla$. Then we have:
%\begin{enumerate}
%\item $g(Z_i, Z_j)= \delta_i ^j$, $i,j=1,2$;
%\item $\nabla _{Z_i} Z_j = 0$, $i,j=1,2$ (in particular the integral curves of the Reeb vector fields are geodesics);
%\item the Reeb action is totally geodesic (i.e the orbits are totally geodesic $2$-dimensional submanifolds).
%\item if $L_{Z_i}$ is the Lie derivative along $Z_i$, then $L_{Z_i}\phi=0$ if and only if $Z_i$ is a Killing vector field.
%\end{enumerate}
%\end{theorem}

\section{Normal metric contact pairs}
We are now interested on metric contact pairs which are at the same time normal.
The following proposition is an immediate corollary of \cite[Corollary 3.2 and Theorem 3.4]{BBH}.

%%%%%%%%%%%%%%%%%%%%%%%%%%%%%%%%%%%%%%%%%%%%%
%
%  Verifier si l'une des proporietes suivantes est utilisee 
%
%%%%%%%%%%%%%%%%%%%%%%%%%%%%%%%%%
\begin{proposition}\label{general-properties}
Let $(\alpha_1,\alpha_2 , \phi , g)$ be a normal MCP on a  manifold, with decomposable $\phi$, Reeb vector fields $Z_1$, $Z_2$, and $Z=Z_1+Z_2$. Let $\nabla$ be the Levi-Civita connection of the associated metric $g$.
% and $\Curv$ its curvature operator, $h_i= \frac{1}{2} L_{Z_i} \phi$ and $h=h_1+h_2$.
%Denote by $d^*$  the codifferential of $d$ with respect to $g$. 
Then we have
\begin{align}
%&h_1=h_2=h=0 \label{eq1-gen-properties}\\
%&\nabla_{Z_i} \phi=0 \, , \, i=1,2 \, ;\label{eq2-gen-properties}\\
&g((\nabla_X \phi)Y, W)=  \sum_{i=1} ^2 [d\alpha_i (\phi Y , X) \alpha_i (W)- d\alpha_i (\phi W,X) \alpha_i (Y)] \, ;\label{eq3-gen-properties}\\
&\nabla_X Z=-\phi X.\label{eq4-gen-properties}%\\
%&d^* (\alpha_1 + \alpha_2)=0 \, ;\label{eq5-gen-properties}\\
%&\Curv_{X Z} Z=X-\alpha_1 (X) Z_1 -\alpha_2 (X) Z_2 \, ;\label{eq6-gen-properties}
%&\Ric (Z)= h+k \, .\label{eq7-gen-properties}
\end{align}
\end{proposition}

Now we want to characterize the normal MCP manifolds between the MCP's as Sasakian manifolds are between the almost contact manifolds.

\begin{theorem}\label{SP-characterisation-connection}
Let $(\alpha_1,\alpha_2 , \phi , g)$ be a contact pair structure on a manifold $M$ with compatible metric $g$, decomposable $\phi$ and Reeb vector fields $Z_1$, $Z_2$.  Then $(\alpha_1,\alpha_2 , \phi , g)$ is a normal MCP if and only if, for all $X,Y \in \Gamma (TM)$,
\begin{equation}\label{eq-SP-characterisation-connection}
(\nabla_X \phi) Y= \sum_{i=1} ^2 [g (X_i, Y_i) Z_i-\alpha_i (Y_i)X_i ] \, ,
\end{equation}
where $X_i$ and $Y_i$ ,  $i=1,2$, are the orthogonal projections of $X$ and $Y$ respectively on the foliation $\mathcal{F}_j$, with $j\neq i$.
\end{theorem}

\begin{proof}
Suppose that $(\alpha_1,\alpha_2 , \phi , g)$ is a normal MCP. By \eqref{eq3-gen-properties}, for all $W\in\Gamma (TM)$ we have:
%if $X,Y$ are tangent to different foliations, since $\phi$ is decomposable, we have $(\nabla_X \phi) Y=0$. Now suppose $X, Y$ tangent to the same foliation let say $\FF_1$ and $W$ tangent to $\FF_2$, then, using again \eqref{eq3-gen-properties} we obtain:
%$$
%g((\nabla_X \phi)Y, W)=  d\alpha_2(\phi Y , X) \alpha_2 (W)- d\alpha_2 (\phi W,X) \alpha_2 (Y)]=0 \, .
%$$
%The same holds for $X, Y$ tangent to $\FF_2$ and $W$ tangent to $\FF_1$. This means that $(\nabla_X \phi) Y$ preserves the foliations exactly as $\phi$ does. For $i=1,2$ we denote by $X_i, Y_i$ the projections of $X,Y$ on the foliation $\FF_j$, $j\neq i$. Then we have
\begin{align*}
g\left(\left(\nabla_X \phi\right) Y,W\right)=&g((\nabla_{X_1} \phi) Y,W)+g((\nabla_{X_2} \phi) Y, W)\\
=&\sum_{i=1} ^2 [d\alpha_i (\phi Y , X_i) \alpha_i (W)- d\alpha_i (\phi W,X_i) \alpha_i (Y)]\\
=&\sum_{i=1} ^2 [(d\alpha_1+d\alpha_2) (\phi Y_i , X_i) \alpha_i (W)- (d\alpha_1+d\alpha_2) (\phi W,X_i) \alpha_i (Y_i)]\\
=&\sum_{i=1} ^2 [g (X_i,Y_i) g (W, Z_i)- \alpha_i (Y_i) g (W,X_i) ]\\
%=&\sum_{i=1} ^2 [g (W_i, g (Y_i , X_i)Z_i)- g (W_i,g (Y_i,Z_i) X_i)]\\
%=&\sum_{i=1} ^2 [g (W_i, g (Y_i , X_i)Z_i)-g (W_i,g (Y_i,Z_i)  X_i)]\\
%=&\sum_{i=1} ^2 g \left(W_i, g (Y_i , X_i)Z_i-g (Y_i,Z_i)  X_i\right)\\
%=&\sum_{i=1} ^2 g \left(W, g ( Y, X_i)Z_i-\alpha_i (Y)  X_i\right)\\
=&g (W,\sum_{i=1} ^2 [g (X_i,Y_i)Z_i-\alpha_i (Y_i)  X_i]) \, ,
\end{align*}
which is equivalent to \eqref{eq-SP-characterisation-connection}.

Conversely, suppose that \eqref{eq-SP-characterisation-connection} is true. Putting $Y=Z_j$ in \eqref{eq-SP-characterisation-connection}, we obtain
%$$(\nabla_X \phi) Z_j=  \alpha_j (X_j) Z_j-X_j \, .$$
%Summing up we have 
$$
-\phi  \nabla_X  Z = (\nabla_X \phi) Z = 
%\sum_{i=1} ^2  (\nabla_X \phi) Z_j= 
\alpha_1 (X) Z_1 + \alpha _2 (X) Z_2 -X= \phi^{2} X\, ,
$$
where $Z=Z_1+Z_2$. This gives 
$
\nabla_X  Z=- \phi X
$
since $\nabla_X  Z$ is horizontal %tangent to the kernels of $\alpha_1$ and $\alpha_2$ 
(see \cite[Lemma 3.5]{BBH}).
Then we have:
%\begin{align*}
%d\alpha_1 (X,Y)&=\frac{1}{2}[X \alpha_1(Y)-Y \alpha_1(X) - \alpha_1 ([X,Y])]\\
%d\alpha_2 (X,Y)&=\frac{1}{2}[X \alpha_2(Y)-Y \alpha_2(X) - \alpha_2 ([X,Y])] 
%\end{align*}
%and then
\begin{align*}
d\alpha_1 (X,Y)+d\alpha_2 (X,Y)&=\frac{1}{2} \sum_{i=1}^2 [X \alpha_i(Y)-Y \alpha_i(X) - \alpha_i ([X,Y])]\\
&=\frac{1}{2} \sum_{i=1}^2 [X g(Z_i ,Y)-Y g(Z_i,X) - g (Z_i, \nabla _X Y - \nabla _Y X)]\\
%&+\frac{1}{2}[X g(Z_2 ,Y)-Y g(Z_2,X) - g(Z_2 ,\nabla _X Y - \nabla _Y X)]\\
&=\frac{1}{2} [g(\nabla _X Z, Y) - g (\nabla _Y Z,X )]\\
%- g (Z_1 ,\nabla _Y X )- g (Z_1, \nabla _X Y - \nabla _Y X)]\\
%&+ \frac{1}{2}[g(\nabla _X Z_2, Y)+ g(Z_2, \nabla _X Y)- g (\nabla _Y Z_2 ,X )- g (Z_2 ,\nabla _Y X )- g (Z_2, \nabla _X Y - \nabla _Y X)]\\
&=\frac{1}{2}[g(-\phi X,Y)+g(\phi Y,X)]\\
&= g (X, \phi Y)  ,
\end{align*}
which means that the compatible metric $g$ is even associated.

To prove the vanishing of the tensor field $N^1$, let us compute $[\phi, \phi]$. Taking $X \in \Gamma (T\FF_1)$ and $Y \in \Gamma (T\FF_2)$ in \eqref{eq-SP-characterisation-connection}, we obtain
$
0=(\nabla_X \phi) Y=\nabla_X (\phi Y)-\phi \nabla_X  Y \, ,
$
which implies $\nabla_X (\phi Y)=\phi \nabla_X  Y$.
Then we obtain
\begin{align*}
[\phi, \phi](X,Y)=&\phi ^2 [X,Y]-\phi  [\phi X,Y]-\phi  [X,\phi Y]+ [\phi X,\phi Y]\\
=&\phi ^2 (\nabla _X Y - \nabla _Y X)-\phi  (\nabla _{\phi X} Y - \phi \nabla _Y X)\\
&-\phi  (\phi \nabla _X Y - \nabla _{\phi Y} X)+ \phi \nabla _{\phi X} Y - \phi \nabla _{\phi Y } X\\
=&0 \, 
\end{align*}
and $N^1(X,Y)=[\phi, \phi](X,Y)=0$.
Now by \eqref{eq-SP-characterisation-connection} with $X, Y\in \Gamma (T\FF_1)$, we have
$
(\nabla_X \phi) Y= g (X, Y) Z_2-\alpha_2 (Y)X \, .
$
Then we get
\begin{align*}
[\phi, \phi](X,Y) =&\phi ^2 (\nabla _X Y - \nabla _Y X)-\phi  (\nabla _{\phi X} Y - \phi \nabla _Y X -(\nabla_Y \phi)X)\\
&-\phi  (\phi \nabla _X Y +(\nabla_X \phi)Y - \nabla _{\phi Y} X)+ \phi \nabla _{\phi X} Y - \phi \nabla _{\phi Y } X\\
=& g(\phi X,Y)Z_2 - g(X,\phi Y)Z_2\\
=&-2 d\alpha_2 (X, Y) Z_2 \, .
\end{align*}
Hence 
$N^1(X,Y)=[\phi, \phi](X,Y)+2 d\alpha_2 (X, Y) Z_2 =0$.
In the same way we obtain 
$N^1(X,Y)=[\phi, \phi](X,Y)+2 d\alpha_1 (X, Y) Z_1=0$  
for all $X, Y\in \Gamma (T\FF_2)$.
This shows the normality and completes the proof.
\end{proof}

\begin{theorem}\label{SP-characterisation-curvature}
Let $(M, \alpha_1,\alpha_2 , \phi , g)$ be an MCP manifold with decomposable $\phi$, $Z_1$, $Z_2$ the Reeb vector fields and $Z=Z_1+Z_2$.  Let $\Curv$ be the curvature operator of $g$.
Then the MCP $(\alpha_1,\alpha_2 , \phi , g)$ is normal if and only if
\begin{equation}\label{eq-SP-character-curv}
\Curv _{X Y} Z=\sum_{i=1} ^2 [\alpha_i (Y_i) X_i-\alpha_i (X_i)Y_i ].
\end{equation}
\end{theorem}
\begin{proof}
Suppose that the MCP is normal. By \eqref{eq4-gen-properties} and \eqref{eq-SP-characterisation-connection}, for all $X,Y \in \Gamma(TM)$, we have
\begin{equation*}
\begin{split}
\Curv_{XY}Z&=-\nabla_X (\phi Y)+\nabla_Y (\phi X)+\phi[X,Y]  \\
                  &=-\big(\nabla_{X}\phi\big)Y+\big(\nabla_{Y}\phi\big)X \\
                  &=\sum_{i=1} ^2 [\alpha_i (Y_i) X_i-\alpha_i (X_i)Y_i ].
\end{split} 
\end{equation*}

Conversely suppose that \eqref{eq-SP-character-curv} is true. Then for $Y$ horizontal %orthogonal to the Reeb vector fields 
we have:
$$
\Curv _{Z Y}Z=-Y_1-Y_2=-Y.
$$
Using this in the following equation %(proved in )
$$
\frac{1}{2} \bigl (\Curv_{Z X}Z - \phi (\Curv_{Z \phi X}Z) \bigr )= \phi^2 X + \h^2 X \, , %\label{th-curv-prop:eq2}.
$$
%\eqref{eq6-gen-properties} 
which holds for MCP manifolds (see \cite[Proposition 4.1]{BBH}),
where $\h= \frac{1}{2} \mathcal{L}_Z \phi$
 and $\mathcal{L}_Z$ is the Lie derivative along $Z$,
we get
$$
\frac{1}{2} (-Y-\phi (-\phi Y))= \phi^2 Y+ \h^2 Y
$$
which implies $\h=0$. In particular from the equation $\nabla_X Z = -\phi X -\phi \h X $  (see \cite[Theorem 3.4]{BBH}), we have $\nabla _X Z=-\phi X$ . Since $\h=0$, the vector field $Z$ is Killing \cite{BBH,BH2}, then it is affine and we have:
$$
\Curv _{ Z X} Y=-\nabla_X \nabla _Y Z+ \nabla _{\nabla_X Y} Z= \nabla _X {\phi Y}- \phi \nabla _X Y= (\nabla_X \phi )Y .
$$
Then for every $X,Y, W \in \Gamma (TM)$, recalling that for an MCP with decomposable $\phi$ the characteristic foliations are orthogonal, we obtain:
\begin{align*}
g ((\nabla_X \phi )Y, W)&= g(\Curv _{Z X}Y , W)= g(\Curv _{Y W }Z , X)= g (\sum_{i=1} ^2 [\alpha_i (W_i) Y_i-\alpha_i (Y_i)W_i ], X)\\
%%%%
&=\sum_{i=1} ^2  [g(\alpha_i (W) Y_i, X_i) - g(W ,  \alpha_i (Y_i) X_i)]\\
&=\sum_{i=1} ^2  [g (W, Z_i)g( Y_i, Xi)-g(W ,  \alpha_i (Y_i) X_i)]\\
&=g (W , \sum_{i=1} ^2  [g( Y_i, X_i) Z_i- \alpha_i (Y_i) X_i]),
\end{align*}
which implies that $(\nabla_X \phi) Y= \sum_{i=1} ^2 [g (X_i, Y_i) Z_i-\alpha_i (Y_i)X_i ]$ and then the pair is normal by Theorem \ref{SP-characterisation-connection}.
\end{proof}

\section{$\phi$-invariant submanifolds}\label{sect:invariant-submanifolds}
In this section we study the $\phi$-invariant submanifolds of MCP manifolds.
We first give some general results%concerning the $\phi$-invariant submanifolds
, then we specialize to several cases concerning the submanifold position relative to the Reeb distribution.

%
%We begin with the following definition:

%\subsection{The case of CP-manifolds}
Let $(\alpha_1, \alpha_2, \phi)$ be a contact pair structure of type $(h,k)$ on a manifold $M$.

\begin{definition}
A submanifold $N$ of $M$ is said to be invariant with respect to $\phi$ (or $\phi$-invariant) if its tangent space at every point is preserved by $\phi$, that is if $\phi_p T_pN \subset T_pN$ for all $p\in N$. 
\end{definition}

In the same way one can define $J$-invariant submanifolds and $T$-invariant submanifolds for the two almost complex structures defined in Section \ref{preliminaries}. 

The simplest examples of $\phi$-invariant surfaces are given by the leaves of the foliation $\mathcal{R}$ tangent to the Reeb distribution.
When we suppose the endomorphism field $\phi$ decomposable, by definition the leaves of the two characteristic foliations $\mathcal{F}_i$ of the $1$-forms $\alpha_i$  are $\phi$-invariant. The same is true for the leaves of the two characteristic foliations $\mathcal{G}_i$ of the 2-forms $d\alpha_i$. 

Observe that in the second case only one of the two Reeb vector fields is tangent to the submanifolds. In the first and third cases both the Reeb vector fields are tangent and such submanifolds are invariant with respect to $J$ and $T$.

Despite to the case of a metric contact manifold, where the Reeb vector field is always tangent to a $\phi$-invariant submanifold, in our case the situation can be quite different as we have just seen. We will show several nontrivial examples in the sequel.

%will show in the sequel.
%\subsection{}%{The case of MCP-manifolds}
In what follows $(M, \alpha_1, \alpha_2, \phi, g)$ will be a given MCP manifold with %decomposable $\phi$ and 
Reeb vector fields $Z_1$, $Z_2$ and $N$ a $\phi$-invariant submanifold of $M$. We will denote by $Z_i^T$ (respectively $Z_i^\bot$) the tangential (respectively normal) component of the two vector fields $Z_1$ and $Z_2$ along $N$.

%Then along $N$ the two vector fields $Z_i $ (for $i=1,2$) decompose as $Z_i=Z_i^T+Z_i^\bot$  where $Z_i^T$ (respectively $Z_i^\bot$)  is tangent (respectively normal) to $N$.

\begin{proposition}\label{fourverticalvectors}
Along the $\phi$-invariant submanifold $N$ the tangent vector fields $Z_1^T$, $Z_2^T$ and the normal vector fields 
$Z_1^\bot$, $Z_2^\bot$ are %pointwise
vertical.%lie in $\R Z_1 \oplus \R Z_2$.
\end{proposition}

\begin{proof}
For every $X\in \Gamma (TN)$ we have $g(\phi Z_i^\bot , X )=-g(Z_i^\bot ,\phi X )=0$ because $\phi X\in \Gamma (TN)$.
Then  the vector fields 
$\phi Z_i^\bot$ are also orthogonal to $N$. 
As $0=(\phi Z_i) _{\vert N}= \phi Z_i^T + \phi Z_i^\bot$, we get $\phi Z_i^T = \phi Z_i^\bot = 0$ because one is tangent and the other is orthogonal to $N$. We conclude by recalling that the distribution $\ker \phi$ is spanned by $Z_1$ and $Z_2$.
\end{proof}

\begin{proposition}\label{z1z2orthogonaux}
There is no point $p$ of the $\phi$-invariant submanifold $N$ such that the tangent vectors $(Z_1)_p$ and $(Z_2)_p$ are both orthogonal to the tangent space $T_p N$.
\end{proposition}

\begin{proof}
If at a point $p\in N$ the two vectors $(Z_i)_p$ (for $i=1,2$) are orthogonal to the tangent space $T_p N$, we have $(Z_i^\bot)_p = (Z_i)_p$ and they are linearly independent. Take an open neighborhood $U$ of $p$ in $M$ such that on $U\cap N$ the two vector fields 
$Z_1^\bot$, $Z_2^\bot$ still remain %pointwise
linearly independent. By Proposition \ref{fourverticalvectors} they span  $\R Z_1 \oplus \R Z_2$ along $U$, and then the Reeb vector fields $Z_1$, $Z_2$ are both orthogonal to $T_q N$ at each point $q\in U\cap N$.

Let $X$ be a vector field defined on $U$, tangent to $N$ and such that $X_p\neq0$.
Then $\phi X$ and $[X,\phi X]$ are also tangent to $N$. Since for every point $q\in U\cap N$ the tangent space $T_q N$ is in the kernels of $\alpha_1$ and $\alpha_2$ (because it is orthogonal to $(Z_i)_q$), along $N$ we have
$$
0=\left(\alpha_1+\alpha_2 \right)([X,\phi X])=-2 \left(d \alpha_1+d \alpha_2 \right)(X, \phi X)=2 g(X,X)
$$
contradicting the fact that $X_p\neq 0$.

\end{proof}

\subsection{The case $N$ tangent to only one Reeb vector field}
%First we start by taking an interest in invariant submanifolds which are tangent to only one of the two Reeb vector fields.

\begin{proposition}

If the $\phi$-invariant submanifold $N$ is tangent to one of the two Reeb vector fields, say $Z_1$, and transverse to the other one $Z_2\,$, then $N$ is everywhere orthogonal to $Z_2\,$. Moreover the dimension of $N$ is odd.

\end{proposition}

Such submanifolds were called semi-invariant by Blair, Ludden and Yano \cite{Blair2} in the context of Hermitian manifolds. The semi-invariance is understood with respect to the almost complex structure $J=\phi  - \alpha_2 \otimes Z_1 + \alpha_1 \otimes Z_2 $.

\begin{proof}

Since $Z_1$ is tangent to $N$, we have 
$Z_1^T=(Z_1)_{\vert N}\neq 0$. Now 
$0=g\left( Z_1, Z_2 \right)_{\vert N}=g( Z_1^T, Z_2 ^T)$ which implies $Z_2^T=0$.
Indeed if at a point $p\in N$, $Z_2^T\neq 0$, the two vectors $(Z_i^T)_p$ would be linearly independent. 
By Proposition \ref{fourverticalvectors}, they will span the tangent subspace $\R (Z_1)_p \oplus \R(Z_2)_p$, and then $Z_2$ will be tangent to $N$ at $p$, but $Z_2$ is supposed to be transverse to N. Now it is clear that $\phi$ is almost complex on the orthogonal complement of $\R Z_1$ in $TN$. Hence the dimension of $N$ is odd.

\end{proof}

The following result is a restatement of \cite[Propositions 4.2 and 4.3]{Blair2}:

\begin{proposition}[\cite{Blair2}]\label{normal_induces_sasaki}
If the $\phi$-invariant submanifold $N$ is tangent to the vector field $Z_1$ and orthogonal to $Z_2$, then $(\alpha_1, \phi, g)$ induces a metric contact structure on $N$. If in addition the MCP on $M$ is normal then $N$ is Sasakian.
\end{proposition}
%???????????????????????

%Peut etre N est tangente au feullietage char. de alpha2 et ensuite sous variete de contact.

%???????????????????????
\begin{proof}
Let $\tilde{\alpha_1}$ and $\tilde{\alpha_2}$ denote the forms induced on $N$ by $\alpha_1$ and $\alpha_2$.
To prove that $\tilde{\alpha_1}$ is a contact form, one has just to show that $\tilde{d\alpha_1}$ is symplectic on $\ker \tilde{\alpha_1}$.
First observe that since $Z_2$ is orthogonal to $N$ we have 
$\tilde{\alpha_2}=0$, then $d\tilde{\alpha_1}=d\tilde{\alpha_1} + d\tilde{\alpha_2}$.
Now, for $p\in N$, let $X\in T_pN$ such that $\tilde{\alpha_1}(X)=0$ and $d\tilde{\alpha_1}(X,Y)=0$  for every $Y\in T_pN$.
Then $(d\tilde{\alpha_1} + d\tilde{\alpha_2})(X,Y)=0$ and we get $g(X,\phi Y)=0$. 
As we also have $0=\alpha_1(X)=g(X,Z_1)$, one can say that $g(X,\cdot)=0$ on $N$ and then $X=0$.
Hence $\tilde{\alpha_1}$ is contact on $N$. 
The normality of the induced structure on $N$ follows from the vanishing of the tensor $N^1$ and the fact that $d\tilde{\alpha_2}=0$ on $N$.

\end{proof}

%\begin{proposition}
%Let $(M, \alpha_1, \alpha_2, \phi, g)$ be a normal MCP manifold with decomposable $\phi$ and Reeb vector fields $Z_1$ and $Z_2$. If $N$ is a $\phi$-invariant submanifold of $M$, tangent to the vector field $Z_1$ and normal to $Z_2$, then the induced metric contact structure on $N$ is Sasakian.
%\end{proposition}
\subsection{The case $N$ nowhere orthogonal and nowhere tangent to $Z_1$ and $Z_2$}

\begin{proposition}\label{prop_Ntransverse}
If the $\phi$-invariant submanifold $N$ is nowhere orthogonal and nowhere tangent to $Z_1$ and $Z_2$, then it has odd dimension. Moreover its tangent bundle $TN$ intersects the vertical subbundle $\mathcal{V}$ along a line bundle spanned by the vector field $Z_1^T$ (or equivalently by $Z_2^T$). 
%If $N$ is tangent to both $Z_1$ and $Z_2$, then it has even dimension.
\end{proposition}

\begin{proof}
In this case, by Proposition \ref{fourverticalvectors} we have necessarily that $Z_1^T$ and $Z_2^T$ are vertical, linearly dependent and nonzero. The same holds for $Z_1^\bot$ and $Z_2^\bot$. 
So $Z_1^T$ spans the intersection of $TN$ with the vertical subbundle $\mathcal{V}$.
Now every vector field tangent to $N$ and orthogonal to $Z_1^T$ is necessary orthogonal to the Reeb distribution. By the $\phi$-invariance of $TN$, $\phi$ is almost complex on the orthogonal complement of $\R Z_1^T$ in $TN$ and then 
$N$ has odd dimension.

%In the second case, the distribution $\R Z_1 \oplus \R Z_2$ is tangent to $N$ and, on its  orthogonal distribution in $TN$ the endomorphism $\phi$ is almost complex. Then $N$ has even dimension.

\end{proof}

\begin{example}\label{exampleh3xh3}
As a manifold consider the product $H^6=\HH_3 \times \HH_3$ where $\HH_3$ is the $3$-dimensional Heisenberg group. Let $\left\{ \alpha_1, \alpha_2, \alpha_3 \right \}$ (respectively $\left \{ \beta_1,  \beta_2,  \beta_3 \right\}$) be a basis of the cotangent space at the identity for the first (respectively second) factor $\HH_3$ satisfying

$$
d\alpha_3=\alpha_1\wedge \alpha_2 \, , \, d\alpha_1=d\alpha_2=0,
$$
$$
d\beta_3=\beta_1\wedge \beta_2 \, , \, d\beta_1=d\beta_2=0.$$
The pair $(\alpha_3 , \beta_3 )$ determines a contact pair of type $(1,1)$ on $H^6$ with Reeb vector fields $(X_3 , Y_3 )$, the $X_i$'s (respectively the $Y_i$'s) being dual to the $\alpha_i$'s (respectively the $\beta_i$'s). The left invariant metric 
$$
g=\alpha_3^2 + \beta_3^2 +\frac {1}{2}(\alpha_1^2 + \beta_1^2 +\alpha_2^2 + \beta_2^2 )
$$
is associated to the pair with decomposable tensor structure $\phi$ given by $\phi(X_2)=X_1$ and $\phi(Y_2)=Y_1$.
The MCP manifold $(H^6,\alpha_3, \beta_3, \phi, g)$ is normal because it is the product of two Sasakian manifolds.
Let $\frak{h}_3$ denotes the Lie algebra of $\HH_3$.
The three vectors $Z=X_3+Y_3$, $X_1 +Y_1$ and $X_2 +Y_2$ span a $\phi$-invariant subalgebra of the Lie algebra 
$\frak{h}_3\oplus \frak{h}_3$ of $H^6$ which determines a foliation in $H^6$. Each leaf is $\phi$-invariant, %submanifold of $H^6$ and 
nowhere tangent and nowhere orthogonal to the Reeb vector fields. 
\end{example}

\subsection{The case $N$ tangent to the Reeb distribution}
\begin{proposition}
If the $\phi$-invariant submanifold $N$ is tangent to both $Z_1$ and $Z_2$, then it has even dimension.
\end{proposition}

\begin{proof}
If the Reeb distribution is tangent to $N$ then on its  orthogonal complement in $TN$ the endomorphism $\phi$ is almost complex and this completes the proof.

\end{proof}

\begin{example}\label{exampleh3xh3bis}
Take the same MCP on $H^6=\HH_3 \times \HH_3$ as in Example \ref{exampleh3xh3}.
The four vectors $X_3$, $Y_3$, $X_1 +Y_1$ and $X_2 +Y_2$ span a  $\phi$-invariant Lie subalgebra $\frak{n}_4$ of $\frak{h}_3\oplus \frak{h}_3$ which determines a foliation on $H^6$. Each leaf %$L$ 
of this foliation is $\phi$-invariant and tangent to the Reeb distribution.
\end{example}

\begin{remark}
When the $\phi$-invariant submanifold $N$ is tangent to both the Reeb vector fields $Z_1$ and $Z_2$, the contact pair $(\alpha_1, \alpha_2)$ on $M$ does not induce necessarily a contact pair on $N$.
Indeed from Example \ref{exampleh3xh3bis}, take any leaf $L^4$ of the foliation determined by the subalgebra $\frak{n}_4$.
Then $L^4$ is a $\phi$-invariant submanifolds of the MCP manifold $H^6$. 
However the contact pair 
$(\alpha_3 , \beta_3 )$ induces a pair of $1$-forms on the $4$-dimensional manifold $L^4$ whose \'Elie Cartan classes are both equal to 3. Then the induced pair on $L^4$ is not a contact one.

\end{remark}

From this construction one can also have a most interesting example where, in addition, the submanifold is closed without being a contact pair submanifold.

\begin{example}\label{exampleh3xh3biscompact}
Consider once again the normal MCP on the nilpotent Lie group $H^6=\HH_3 \times \HH_3$ defined in Example \ref{exampleh3xh3},
and the foliation on $H^6$ defined by the Lie algebra $\frak{n}_4$ described in Example \ref{exampleh3xh3bis}.
Let $L_e$ be the leaf passing through the identity element of  the Lie group $H^6$. One can see that the Lie subgroup $L_e$ is isomorphic to $\HH_3 \times \mathbb{R}$. In fact by using the change of basis of its Lie algebra $\frak{n}_4$, $U_i=X_i+Y_i$ for $i=1,2,3$ and $U_4=X_3$, we get $[U_1,U_2]=U_3$ and the other brackets are zero. Since the structure constants of the nipotent Lie algebra $\frak{n}_4$ are rational, there exist cocompact lattices $\Gamma$ of $L_e$. For example, since $\mathbb H_3$ can be considered as the group of the real matrices 
$$
\gamma(x,y,z)=
\begin{pmatrix}
1 & y & z \\
0 & 1 & x \\
0 & 0 & 1
\end{pmatrix},
$$
take $\Gamma \simeq \Gamma_r \times \mathbb{Z}$ where $\mathbb{Z}$ acts on the factor $\mathbb{R}$ and 
$\Gamma_r=\{\gamma(x,y,z) \vert x\in \mathbb Z,y\in r\mathbb Z, z\in \mathbb Z\}$, with $r$ a positive integer, acts on the first factor $\mathbb{H}_3$ by left multiplication (see e.g. \cite{GW}).
Because $L_e$ is a subgoup of $H^6$, it is a lattice of $H^6$ too. Now the closed nilmanifold $N^4=L_e/\Gamma$ is a submanifold of  the nilmanifold $M^6= H^6 /\Gamma$. Since the MCP on $H^6$ is left invariant, it descends to the quotient $M^6$ as a normal MCP 
$(\tilde{\alpha_3}, \tilde{\beta_3}, \tilde{\phi}, \tilde{g})$ 
of type $(1,1)$ with decomposable endomorphism $\tilde{\phi}$. Moreover the closed submanifold $N^4$ is $\tilde{\phi}$-invariant and tangent to the Reeb distribution. Note that the contact pair $(\tilde{\alpha_3}, \tilde{\beta_3})$ on $M^6$ does not induce a contact pair on the submanifold $N^4$, because the \'Elie Cartan classes of the induced $1$-forms are equal to $3$ and the dimension of $N^4$ is $4$.
\end{example}

We know (see \cite{BH3}) that a normal MCP with decomposable endomorphism is nothing but a Hermitian bicontact manifold of bidegree $(1,1)$ \cite{Blair2}.
As we will see later in Paragraph \ref{paragraph-J-invariance}, in a normal MCP a $\phi$-invariant submanifold tangent to the Reeb distribution is a complex submanifold.
So according to Example \ref{exampleh3xh3biscompact} we can state what follows:

\begin{proposition}\label{prop:abe-conterexample}
There exists a Hermitian bicontact manifold of bidegree $(1,1)$ carrying a closed complex submanifold which does not inherit a bicontact structure.
\end{proposition}

\begin{remark}
This contradicts a statement of Abe (see \cite[Theorem 2.2]{Abe}).
The construction of the MCP manifold $M^6$ and its submanifold $N^4$ in Example \ref{exampleh3xh3biscompact} gives clearly a counterexample.
\end{remark}

\subsection{Relationship with $T$ and $J$-invariance}\label{paragraph-J-invariance}
Put $\rho= \alpha_2 \otimes Z_1 - \alpha_1 \otimes Z_2 $. One can easily see that a connected submanifold of $M$ is $\rho$-invariant if and only if it is tangent or orthogonal to the Reeb distribution. The following holds:

\begin{proposition}\label{prop-N-orthogonale}
Let $M'$ be a submanifold of the MCP manifold $M$. If $M'$ is orthogonal to the Reeb distribution, then none of the endomorphisms $\phi$, $J$ and $T$ leaves $M'$ invariant.
\end{proposition}

\begin{proof}
Let $M'$ be a submanifold of $M$ orthogonal to the Reeb distribution. By Proposition \ref{z1z2orthogonaux}, it cannot be $\phi$-invariant.
Suppose that $M'$ is invariant with respect to $J$ or $T$. Since it is orthogonal to the Reeb vector fields, it is also $\rho$-invariant. Now by the relations $\phi=J+\rho=T-\rho$, we obtain that $M'$ is $\phi$-invariant, and this is not possible.
\end{proof}

\begin{proposition} \label{prop-equiv-JTphi}
Let $M'$ be a submanifold of the MCP manifold $M$. Then any two of the following 
%Consider the following 
properties imply the others:
\begin{enumerate}%[label=\alph*.]%[a]
\item[(a)] $M'$ is $\phi$-invariant,
\item[(b)] $M'$ is $J$-invariant,
\item[(c)] $M'$ is $T$-invariant,
\item[(d)]  $M'$ is tangent to the Reeb distribution.
\end{enumerate}
%Then any two of these four properties 

\end{proposition}

\begin{proof}
From the relations $J=\phi-\rho$ and $T=\phi+\rho$, one can remark that any two of the four endomorphisms fields $\{\phi, J, T, \rho \}$ are linear combinations of the remaining two. So if we replace the property (d) with ``$M'$ is $\rho$-invariant'', then the conclusion is obvious. Suppose without loss of generality that $M'$ is connected. We have seen that the property ``$M'$ is $\rho$-invariant" is equivalent to ``$M'$ is tangent or orthogonal to the Reeb vector fields''. But by Proposition \ref{prop-N-orthogonale}, the property ``$M'$ is orthogonal to the Reeb distribution'' is not compatible with Properties (a) , (b) and (c), and this completes the proof.

\end{proof}

\begin{example}\label{example_J_notfi_invariant}
Consider $\mathbb{R}^{2h+2k+2}$ together with the normal MCP described in Example \ref{exampledarboux} with $h>0$. For any pair of integers $n_1, n_2$ such that $0< n_1\leq h$ and $0\leq n_2\leq k$, the $2(n_1+n_2)$-dimensional distribution spanned by the vector fields $Y_i=X_i+\frac{1}{2}x_iZ_1$, $JY_i$ for $i=1, \ldots ,n_1$ and (in the case $n_2 > 0$) by $Y^{\prime}_j=X^{\prime}_j+\frac{1}{2}x^{\prime}_jZ_2$, $JY^{\prime}_j$ for $j=1, \ldots ,n_2$ is completely integrable. On the open set $\{ x_i\neq 0, x^\prime_j\neq 0\}$ this distribution is invariant with respect to the complex structure $J$ but it is not invariant with respect to $\phi$. So it gives rise to a foliation by $2(n_1+n_2)$-dimensional  $J$-invariant submanifolds which are not $\phi$-invariant.
\end{example}

\section{Minimal $\phi$-invariant submanifolds}

In Section \ref{sect:invariant-submanifolds}, we observed that the leaves of the two characteristic foliations of an MCP with decomposable endomorphism $\phi$ are $\phi$-invariant submanifolds. Moreover, in \cite{BH4} we have seen that these submanifolds are minimal.
In this section, we extend this result to further $\phi$-invariant submanifolds of normal or {\it complex MCP} manifolds (the latter  terminology meaning that just one of the two natural almost complex structure is supposed to be integrable).

\begin{theorem}\label{th:phi-invariant Z_1 normal}
Let $(M, \alpha_1, \alpha_2, \phi, g)$ be a normal MCP manifold with decomposable $\phi$ and Reeb vector fields $Z_1$ and $Z_2$.
If $N$ is a $\phi$-invariant submanifold of $M$ such that $Z_1$ is tangent and $Z_2$ orthogonal to $N$, then $N$ is minimal. Moreover if $N$ is connected, then it is a Sasakian submanifold of one of the Sasakian leaves of the characteristic foliation of $\alpha_2$.
\end{theorem}

\begin{proof}
Denote by $\B$ the second fundamental form of the submanifold $N$, by 
$\nabla$  the Levi-Civita connection of the metric $g$ on $M$, and by $\tilde{\nabla}$ its induced connection on $N$. By Proposition \ref{normal_induces_sasaki}, $(\alpha_1,Z_1,\phi,g)$ induces a Sasakian structure on $N$, say $(\tilde{\alpha_1},Z_1,\tilde{\phi},\tilde{g})$.
Then for every $X$, $Y\in \Gamma(TN)$ we have $\left( \tilde{\nabla}_X \tilde{\phi} \right)Y=\tilde{g}(X,Y)Z_1-\tilde{\alpha_1}(Y)X$ (see e.g. \cite{Blairbook}).
Using this and \eqref{eq-SP-characterisation-connection} %, for every vector fields $X$, $Y\in \Gamma(TN)$ orthogonal to $Z_1$ 
for all $X$, $Y\in \Gamma(TN)$ orthogonal to $Z_1$,
we obtain

\begin{equation}\label{2nd_Fund_form_phi}
\B(X,\phi Y)-\phi \B(X,Y)=\left( \nabla_X \phi \right)Y-\left( \tilde{\nabla}_X \tilde{\phi} \right)Y=g(X_2,Y_2)(Z_2-Z_1)
\end{equation}
since $X$, $Y$ are 
 horizontal, 
 $X_2$ and $Y_2$ being respectively the orthogonal projections of $X$ and $Y$  on $T\FF_1$.
 But the vector field 
$\B(X,\phi Y)-\phi \B(X,Y)$ must be orthogonal to $N$ by the $\phi$-invariance of $N$. Then $g(X_2,Y_2)=0$ for every $X$, $Y\in \Gamma(TN)$ orthogonal to $Z_1$, which gives $X_2=0$. This implies that $N$ is tangent to the characteristic distribution of $\alpha_2$.

Equation \eqref{2nd_Fund_form_phi} becomes
$$
\B(X,\phi Y)-\phi \B(X,Y)=0.
$$
If we interchange the roles of $X$ and $Y$ and take the difference, we get $\B(X,\phi Y)=\B(Y,\phi X)$ which implies
$\B(X,Y)=-\B(\phi X,\phi Y)$. Now locally, take an orthonormal  $\phi$-basis of the metric contact structure on $N$
$$
Z_1,e_1, \phi e_1, \ldots , e_s , \phi e_s.
$$
We have $\B(Z_1,Z_1)=0$ since $\nabla_{Z_1}Z_1=0$ (see \cite{BH2}).
%and $\tilde{\nabla}_{Z_1}Z_1=0$ which is a well known property of the Reeb vector field of a metric contact structure (see e.g. \cite{Blairbook}).
As $e_j$ are orthogonal to  $Z_1$, we obtain:
$$
%\tr  \meancurv_\eta=g(\B(Z_1 ,Z_1),\eta)+\sum_{j=1} ^s\left( g(\B(e_j,e_j),\eta)\right)+ g(\B(\phi e_j,\phi e_j),\eta)=0,
\tr (\B)=\B(Z_1 ,Z_1)+\sum_{j=1} ^s\left( \B(e_j,e_j)+ \B(\phi e_j,\phi e_j)\right)=0,
$$
which means that $N$ is minimal.
\end{proof}

Consider a $\phi$-invariant submanifold $N$ of an MCP which is nowhere tangent and nowhere orthogonal to the Reeb vector fields.
In Proposition \ref{prop_Ntransverse}, we have seen that at every point its tangent space intersects the Reeb distribution giving rise to the distribution on $N$ spanned by $Z_1^T$ (or equivalently by $Z_2^T$). For such a submanifold we have

\begin{theorem}\label{th:phi-invariant-transverse}
Let $(M, \alpha_1, \alpha_2, \phi, g)$ be a normal metric contact pair manifold with decomposable $\phi$ and Reeb vector fields $Z_1$ and $Z_2$.
Let $N$ be a $\phi$-invariant submanifold of $M$ nowhere tangent and nowhere orthogonal to $Z_1$ and $Z_2$.
Then $N$ is minimal if and only if the angle between $Z_1^T$ and $Z_1$ (or equivalently $Z_2$) is constant along the integral curves of $Z_1^T$.
\end{theorem}

\begin{proof}
Put $\zeta=\frac{1}{\Vert Z_1^T\Vert}Z_1^T=\pm\frac{1}{\Vert Z_2^T\Vert}Z_2^T$. Using \eqref{eq-SP-characterisation-connection} %, for every vector fields $X$, $Y\in \Gamma(TN)$ orthogonal to $Z_1$ 
for all $X$, $Y\in \Gamma(TN)$ orthogonal to $\zeta$,
we obtain

\begin{equation*}\label{2nd_Fund_form_phi}
\B(X,\phi Y)-\phi \B(X,Y)=\left(\left( \nabla_X \phi \right)Y\right)^\bot =g(X_1,Y_1)Z_1^\bot+g(X_2,Y_2)Z_2^\bot
\end{equation*}
since $X$, $Y$ are %pointwise
horizontal % $\in \ker \alpha_1\cap \ker\alpha_2\,$ 
because they are necessarily orthogonal to $Z_1$ and $Z_2$. The term on the right  is symmetric on $(X,Y)$, then we get
$$
\B(X,\phi Y)-\B( \phi X,Y)=0.
$$
As previously this yields %If we interchange the roles of $X$ and $Y$ and take the difference, we get 
%$\B(X,\phi Y)=\B(Y,\phi X)$ which implies
$\B(X,Y)=-\B(\phi X,\phi Y)$.
Now take a local orthonormal basis on $N$ in this manner
$$
\zeta,e_1, \phi e_1, \ldots , e_s , \phi e_s.
$$
We obtain
$$
\tr (\B)=\B(\zeta ,\zeta)+\sum_{j=1} ^s\left( \B(e_j,e_j)+ \B(\phi e_j,\phi e_j)\right)=\B(\zeta ,\zeta).
$$

In order to compute $\B(\zeta ,\zeta)$, 
observe that there exists a smooth function $\theta$ on $N$ taking nonzero values in $] -\pi/2,\pi/2[$ and for which 
$\zeta=(\cos\theta) Z_1+(\sin\theta) Z_2$.
This function is well defined on $N$ since $\zeta$ lies in $\mathcal{V}=\mathbb{R} Z_1 \oplus \mathbb{R} Z_2$ and $g(\zeta, Z_1)>0$. It represents the oriented angle $(Z_1,\zeta)$ in the oriented orthonormal basis $(Z_1,Z_2)$ of  $\mathcal{V}$ along $N$. One can easily show that 
$Z_1^T=(\cos \theta) \zeta$, $Z_2^T=(\sin \theta) \zeta$ and then, since $J\zeta=-(\sin\theta) Z_1+ (\cos\theta)Z_2$, we have $Z_1^\bot=-(\sin \theta) J\zeta$ and $Z_2^\bot=(\cos \theta) J\zeta$. Hence $J\zeta$ is a nonvanishing section of the normal bundle $TN^\bot$ of $N$ in $M$.

Using the equations $\nabla_{Z_i}Z_j=0$, for $i,j=1,2$, concerning  MCP's \cite{BH2}, we obtain 
$$
\nabla_\zeta \zeta=\zeta(\theta)J\zeta .
$$
This yields
$$
\tr (\B)=\B(\zeta,\zeta)=\zeta(\theta)J\zeta
$$
which is zero if and only if $\zeta(\theta)=0$.
\end{proof}
\begin{example}
For the $\phi$-invariant submanifolds described in Example \ref{exampleh3xh3}, the Reeb vector fields $X_3$ and $Y_3$ make a constant angle with their orthogonal projection $\frac{1}{2}(X_3+Y_3)$. Hence they are minimal.
\end{example}

A theorem of Vaisman \cite{Vaisman} states that on a Vaisman manifold a complex submanifold inherits the structure of Vaisman manifold if and only if it is minimal or equivalently if and only if it is tangent to the Lee vector field (and therefore tangent to the anti-Lee one). This result has been generalized to the lcK manifolds as follows (see \cite[Theorem 12.1]{Ornea}): a complex submanifold of an lcK manifold is minimal if and only if it is tangent to the Lee vector field. Non-K\"ahler Vaisman manifolds are special lcK manifolds. According to \cite{BK2} they are, up to a constant rescaling of the metric, exactly normal MCP manifolds of type $(h,0)$, the Reeb vector fields being the Lee and the anti-Lee vector field. What follows is a generalization of the theorem of Vaisman to complex MCP manifolds of any type $(h,k)$.

\begin{theorem}\label{th-j-integrable-minimal}
Let $(M, \alpha_1, \alpha_2, \phi, g)$ be an MCP manifold with decomposable $\phi$ and Reeb vector fields $Z_1$ and $Z_2$. Suppose that the almost complex structure  $J=\phi  - \alpha_2 \otimes Z_1 + \alpha_1 \otimes Z_2 $ is integrable. Then a $J$-invariant submanifold $N$ of $M$ is minimal if and only if it is tangent to the Reeb distribution.
\end{theorem}

We have the same conclusion if we replace  $J$ with the almost complex structure $T=\phi  + \alpha_2 \otimes Z_1 - \alpha_1 \otimes Z_2 $. Recall that for a submanifold tangent to the Reeb distribution, we have equivalence between invariance with respect to $J$, $T$ and $\phi$ (see Proposition \ref{prop-equiv-JTphi}).
Hence minimal $J$-invariant submanifolds of an MCP  are necessarily $\phi$-invariant and $T$-invariant. We can restate Theorem \ref{th-j-integrable-minimal} for a normal MCP as follows:
\begin{corollary}
Let $(M, \alpha_1, \alpha_2, \phi, g)$ be a normal MCP manifold with decomposable $\phi$. Then a $J$-invariant submanifold $N$ of $M$ is minimal if and only if it is $T$-invariant.
\end{corollary}

The $J$-invariant submanifolds described in Example \ref{example_J_notfi_invariant} are not minimal. Of course they are not tangent to the Reeb distribution.

We have seen that the MCP on $\HH_3 \times \HH_3$, given in Example \ref{exampleh3xh3}, is normal because each factor is a Sasakian manifold. The submanifolds decribed in Example \ref{exampleh3xh3bis} are tangent to the Reeb distribution and then they are minimal. The following statement gives further interesting examples.

\begin{corollary}
Consider an MCP $( \alpha_1, \alpha_2, \phi, g)$ with decomposable $\phi$ on a manifold. Suppose that $J$ (or $T$) is integrable. Then the leaves of the characteristic foliations $\mathcal{G} _1$ and $\mathcal{G} _2$ of $d\alpha_1$ and $d\alpha_2$ are minimal.
\end{corollary}

\begin{proof}[Proof of Theorem \ref{th-j-integrable-minimal}] 

In order to compute the normal mean curvature $H$ of the $J$-invariant submanifold $N$, one needs the expression of the tensor field $F(X,Y)=(\nabla_X J)Y$ where $\nabla$ is the Levi-Civita connection of $g$. Since $J$ is integrable, $g$ is Hermitian with fundamental $2$-form 
$$
\Phi =d\alpha_1+d\alpha_2-2\alpha_1\wedge \alpha_2.
$$
First observe that 
$\alpha_2 \circ J=\alpha_1$ and $d\alpha_i(JX,JY)=d\alpha_i(X,Y)$.
Moreover by the decomposability of $\phi$ we have
$\pi_i \circ J=J\circ \pi_i$ where 
%$\pi_i:ª\rightarrow T\mathcal{ G}_j$
$\pi_i:TM\rightarrow \mathcal{H}_j$ (for $j\neq i$ with $i,j=1,2$) denote the orthogonal projections.
Next, using this and the classical equation for a Hermitian structure
$4g((\nabla_X J)Y,W)=6d\Phi(X,JY,JW)-6d\Phi(X,Y,W)$, after a straightforward calculation we get:
\begin{equation*}
\begin{split}
F(X,Y)&=
[-d\alpha_2(X,Y)-d\alpha_1(X,JY)]Z_1
+[d\alpha_1(X,Y)-d\alpha_2 (X, JY)]Z_2 \\
&+\alpha_2(Y)\pi_1 JX-\alpha_1(Y) \pi_2 JX-\alpha_1(Y) \pi_1 X -\alpha_2 (Y) \pi_2 X.
\end{split}
\end{equation*}
Any vector $v$ tangent to $M$ at a point of $N$ decomposes as $v=v^T+v^\bot$ where $v^T$ and $v^\bot$ are tangent and orthogonal to $N$ respectively. The $J$-invariance implies that $J(v^T)=(Jv)^T$ and $J(v^\bot)=(Jv)^\bot$.
Denote by $\operatorname{B}$ the second fundamental form of the submanifold $N$.
Then
$
\operatorname{B}(X,JY)=J\operatorname{B}(X,Y)+F(X,Y)^\bot
$
and
$
\operatorname{B}(JX,JY)+\operatorname{B}(X,Y)=JF(Y,X)^\bot+F(JX,Y)^\bot
$.
Hence we obtain
\begin{equation*}
\begin{split}
\operatorname{B}(X,X)+\operatorname{B}(JX,JX)&=
-2\lVert \pi_2 X \rVert^2 Z_1^\bot +2\lVert \pi_1 X \rVert^2Z_2^\bot\\
&+2[-\alpha_1(X)\pi_1 JX-\alpha_2(X) \pi_2 JX-\alpha_2(X) \pi_1 X +\alpha_1 (X) \pi_2 X]^\bot\
\end{split}
\end{equation*}
Let $N^\prime$ be the open set of $N$ consisting of all points where $Z_1^T\neq0$. It is also a $J$-invariant submanifold of $M$. Take an orthonormal (local) $J$-basis on $N^\prime$
$$
e_1, J e_1, \ldots , e_n , J e_n.
$$
One can choose it in such a way that 
$e_1=\frac{1}{\lVert Z_1^T\rVert}Z_1^T$
and then
$Je_1=\frac{1}{\lVert Z_1^T\rVert}Z_2^T$.
Since the $e_l$ and $Je_l$, for $l=2, \ldots ,n$, are orthogonal to $Z_1^T$, $Z_2^T$, $Z_1^\bot$ and $Z_2^\bot$, they are orthogonal to $Z_1$ and $Z_2$ too. So they are 
 horizontal. 
Now the normal mean curvature along $N^\prime$ is
$$H_{\vert N^\prime}=\frac{1}{2n}\tr (\operatorname{B})=\frac{1}{2n}\sum_{l=1}^n\left(\operatorname{B}(e_l,e_l)+\operatorname{B}(Je_l,Je_l)\right)$$
and becomes
\begin{equation}\label{eq-meancurvature}
H_{\vert N^\prime}=\frac{1}{n}\left(-\sum_{l=1}^n \lVert\pi_2 e_l\rVert^2 Z_1^\bot +\sum_{l=1}^n \lVert\pi_1 e_l\rVert^2 Z_2 ^\bot +(\pi_2Z_1^T -\pi_1 Z_2^T)^\bot\right).
\end{equation}

If we suppose $N$ minimal, then $H_{\vert N^\prime}=0$ and the scalar products with $Z_i^\bot$ yield
\begin{equation}\label{eq-scalarproductZH}
\begin{split}
0&=ng(H_{\vert N^\prime},Z_1^\bot)= -\lVert\pi_2 Z_1^T\rVert^2   -  \lVert Z_1^\perp\rVert^2  \sum_{l=1}^n \lVert \pi_2 e_l\rVert ^2 \\
0&=ng(H_{\vert N^\prime},Z_2^\bot)= \lVert\pi_1 Z_2^T\rVert^2   + \lVert Z_2^\perp\rVert^2  \sum_{l=1}^n \lVert \pi_1 e_l\rVert ^2.
\end{split}
\end{equation}
%%%%%%%%%%%%%%%%%%%%%%%%%%%
Indeed, using \eqref{eq-meancurvature} and the fact that $Z_2 ^\bot=(J Z_1 )^\bot=JZ_1 ^\bot$,
we have 
$$
ng(H_{\vert N^\prime},Z_1^\bot)=  g( (\pi_2Z_1^T -\pi_1 Z_2^T)^\bot , Z_1^\bot )
-  \lVert Z_1^\perp\rVert^2  \sum_{l=1}^n \lVert \pi_2 e_l\rVert ^2 $$
The first term on the right becomes
\begin{equation*}
\begin{split}
g( (\pi_2Z_1^T -\pi_1 Z_2^T)^\bot , Z_1^\bot )&=g( \pi_2Z_1^T -\pi_1 Z_2^T, Z_1-Z_1^T )\\
                                                                        &=-g( \pi_2Z_1^T -\pi_1 Z_2^T, Z_1^T )\\                                                                      
                                                                        &=-\lVert\pi_2 Z_1^T\rVert^2 + g(\pi_1 Z_2^T, \pi_1 Z_1^T )\\
                                                                        &=-\lVert\pi_2 Z_1^T\rVert^2 \, ,
\end{split}
\end{equation*}
where we have used the pairwise orthogonality of the subbundles $\mathcal{H}_1$, $\mathcal{H}_2$, $\mathbb{R} Z_1$, $\mathbb{R} Z_2$, and for the last simplification the fact that $\pi_1 Z_2^T=\pi_1 (J Z_1)^T=J\pi_1 Z_1^T$.
Thus we obtain the first equation of \eqref{eq-scalarproductZH}. The second one comes after an analogous computation.
%%%%%%%%%%%%%%%%%%%%%%%%%%%

%Then
Now from Equations \eqref{eq-scalarproductZH}, %%%
we get $\pi_2 Z_1^T=0$ and $\pi_1 Z_1^T=J\pi_1 Z_2^T=0$ along $N^\prime$, which means that $Z_i^T=Z_i$ at these points. Hence $Z_1$ and $Z_2$ are tangent to 
$N^\prime$. Now we have to prove that $N=N^\prime$. Every point $p$ of $N$ is in the closure of $N^\prime$ in $N$. For otherwise, there exists an open  neighborhood $U_p$ of $p$ in $N$ which does not intersect $N^\prime$, i.e. $Z_1$ and $Z_2$ are orthogonal to the $J$-invariant (and then also $\phi$-invariant) submanifold $U_p$. But this contradicts Proposition \ref{z1z2orthogonaux}. Now since $Z_1^T=Z_1$ on $N^\prime$, by continuity of $Z_1^T$ we have $(Z_1^T)_p=(Z_1)_p$ and then $p\in N^\prime$. Hence $N=N^\prime$ so that $Z_1$ and $Z_2=JZ_1$ are tangent to $N$.

Conversely suppose that $Z_1$ and $Z_2$ are tangent to $N$. Then $Z_i^T=(Z_i)_{\vert N}$ which implies that $N=N^\prime$, and replacing in Equation \eqref{eq-meancurvature}  we get  $H=0$. This completes the proof.
 
\end{proof}
\begin{remark}
One could hope on a full generalization of the original Vaisman result, which could be stated as follows: a $J$ and $T$-invariant submanifold of a normal MCP inherits the structure of normal MCP if and only if it is minimal. In fact this kind of generalization is not possible, because the submanifold in Example \ref{exampleh3xh3biscompact} is both $J$ and $T$-invariant, therefore it is minimal, but it does not inherit the normal MCP of the ambient manifold by Proposition \ref{prop:abe-conterexample}.
\end{remark}

\section*{Acknowledgements}
The authors are grateful to Paola Piu and Michel Goze for their helpful suggestions. They also wish to thank the referee for his useful comments that improved the paper.

\end{document}